\documentclass{article}
\usepackage[T1]{fontenc}
\usepackage[utf8]{inputenc} 
\usepackage{lmodern}
\usepackage[english]{babel}
\usepackage[
  letterpaper,
  top=2.5cm,
  bottom=3.2cm,   
  left=3cm,
  right=3cm,
  marginparwidth=1.75cm,
  footskip=15mm   
]{geometry}
\usepackage{microtype}
\usepackage{amsmath,amsthm,amssymb,amsfonts}
\theoremstyle{plain}
\newtheorem{theorem}{Theorem}[section]
\newtheorem{prop}[theorem]{Proposition}
\newtheorem{cor}[theorem]{Corollary}
\newtheorem{lemma}[theorem]{Lemma}
\theoremstyle{definition}

\newtheorem{remark}[theorem]{Remark}
\newtheorem{example}[theorem]{Example}
\usepackage{authblk}
\newcommand{\T}{\mathcal{T}}
\newcommand{\Aut}{\operatorname{Aut}}
\usepackage[colorlinks=true, allcolors=blue]{hyperref}
\title{Isomorphisms, Moduli, and Cohomological Dimension\\ for Twisted Triangular Banach Algebras}
\author[1]{Sara Behnamian}
\author[2]{Fatemeh Fogh}
\affil[1]{Globe Institute, University of Copenhagen, \mbox{Øster Voldgade 5--7}, 1350 Copenhagen K, Denmark\\
\texttt{sara.behnamian@sund.ku.dk}}
\affil[2]{Department of Mathematical Sciences, Florida Atlantic University, Boca Raton, FL 33431, USA\\
\texttt{ffogh2021@fau.edu}}
\date{} 
\begin{document}
\maketitle
\begin{abstract}
We introduce and study \emph{twisted triangular Banach algebras} 
\(\T_\sigma(A,B;X)\), built from Banach algebras \(A,B\), a Banach 
\(A\)--\(B\) bimodule \(X\), and a pair of automorphisms 
\(\sigma=(\sigma_A,\sigma_B)\). This construction extends the classical 
triangular framework by incorporating twisted module actions on the 
off--diagonal block. We obtain a complete isomorphism classification: 
\(\T_\sigma\simeq\T_\tau\) precisely when the diagonal twists are conjugate, 
the bimodule admits an \((\alpha,\beta)\)--equivariant isomorphism, and a shear 
cocycle satisfies a natural identity. In the case of group algebras, the 
classification detects conjugacy classes inside \(\Aut(G)\), yielding new 
dynamical invariants absent from the untwisted setting. On the homological 
side, we establish sharp bounds for the Hochschild cohomological dimension and 
deduce that \(\T_\sigma\) is amenable only when \(X=0\) and both $A,B$ are 
amenable. Thus twisting enriches the classical theory while preserving 
extension--theoretic control of cohomology.
\end{abstract}

\begingroup
\renewcommand\thefootnote{}%

\addtocounter{footnote}{-1}%
\endgroup


\section{Preliminaries}

Throughout we fix a scalar field \(\mathbb{K}\in\{\mathbb{R},\mathbb{C}\}\).
All algebras are Banach \(\mathbb{K}\)-algebras, all homomorphisms are bounded and \(\mathbb{K}\)-linear, and all tensor products are understood in the completed projective sense unless stated otherwise.
We follow Helemskii’s conventions for continuous Hochschild cohomology and related homological notions \cite{HelemskiiHomology}.

Let \(A\) and \(B\) be Banach algebras and let \(X\) be a Banach \(A\)--\(B\) bimodule with bounded left and right actions.
For a pair of algebra automorphisms \(\sigma=(\sigma_A,\sigma_B)\in \Aut(A)\times \Aut(B)\) we define the \emph{twisted triangular Banach algebra}
\[
\T_\sigma(A,B;X)
=\left\{\begin{bmatrix}a&x\\[2pt]0&b\end{bmatrix}: a\in A,\ x\in X,\ b\in B\right\},
\qquad
\left\|\begin{bmatrix}a&x\\[2pt]0&b\end{bmatrix}\right\|=\|a\|+\|x\|+\|b\|,
\]
with multiplication
\[
\begin{bmatrix}a&x\\[2pt]0&b\end{bmatrix}
\cdot_\sigma
\begin{bmatrix}a'&x'\\[2pt]0&b'\end{bmatrix}
=
\begin{bmatrix}
aa' & \sigma_A(a)\cdot x' + x\cdot \sigma_B(b')\\[2pt]
0   & bb'
\end{bmatrix}.
\]
Associativity of \(\cdot_\sigma\) follows from the bimodule axioms and multiplicativity of \(\sigma_A,\sigma_B\): for arbitrary triples one has
\[
\bigl((a,x,b)\cdot_\sigma(a',x',b')\bigr)\cdot_\sigma(a'',x'',b'')
=
(a,x,b)\cdot_\sigma\bigl((a',x',b')\cdot_\sigma(a'',x'',b'')\bigr),
\]
since the off-diagonal term on each side equals
\(\sigma_A(a)\cdot(\sigma_A(a')\cdot x'' + x'\cdot\sigma_B(b'')) + x\cdot \sigma_B(b'b'')\).
Boundedness of the product with respect to the \(\ell^1\)-type norm is immediate from the bounded module actions and the boundedness of \(\sigma_A,\sigma_B\).
When \(\sigma=(\mathrm{id}_A,\mathrm{id}_B)\) the construction reduces to the usual triangular Banach algebra \cite{ForrestMarcouxWA,MoslehianTriangular}; the twisted off-diagonal interaction is in the spirit of \cite{BehnamianMahmoodi2020}. See also \cite{BehnamianMahmoodi2021} for Arens regularity of triangular algebras with modified multiplications.

An essential structural feature of \(\T_\sigma\) is the closed ideal
\[
I=\left\{\begin{bmatrix}0&x\\[2pt]0&0\end{bmatrix}: x\in X\right\},
\]
which satisfies \(I^2=0\).
The diagonal quotient is naturally
\[
Q=A\oplus_1 B \ \cong\ \left\{\begin{bmatrix}a&0\\[2pt]0&b\end{bmatrix}: a\in A,\ b\in B\right\},
\]
and there is a short exact sequence of Banach algebras
\[
0\longrightarrow I \longrightarrow \T_\sigma \xrightarrow{\ \pi\ } Q \longrightarrow 0,
\qquad
\pi\!\begin{bmatrix}a&x\\[2pt]0&b\end{bmatrix}=(a,b).
\]
If \(A\) and \(B\) are unital we write
\(
e_{11}=\begin{bmatrix}1_A&0\\[2pt]0&0\end{bmatrix}
\)
and
\(
e_{22}=\begin{bmatrix}0&0\\[2pt]0&1_B\end{bmatrix}.
\)
When \(A,B\) merely have bounded approximate identities, we tacitly work in the multiplier algebras to speak of these idempotents.
In either case one has the Peirce decomposition
\[
\T_\sigma = e_{11}\T_\sigma e_{11}\ \oplus\ e_{11}\T_\sigma e_{22}\ \oplus\ e_{22}\T_\sigma e_{22},
\]
with \(e_{11}\T_\sigma e_{11}\cong A\), \(e_{22}\T_\sigma e_{22}\cong B\), \(e_{11}\T_\sigma e_{22}\cong X\), and \(e_{22}\T_\sigma e_{11}=0\).

Given \(\alpha\in\Aut(A)\) and \(\beta\in\Aut(B)\), a bounded bijection \(u:X\to X\) is an \(A\)--\(B\) bimodule isomorphism \emph{relative to \((\alpha,\beta)\)} if
\[
u(a\cdot x\cdot b)=\alpha(a)\cdot u(x)\cdot\beta(b)\qquad(a\in A,\ b\in B,\ x\in X).
\]
The set of such maps forms a group \(\Aut_{A,B}(X)\) under composition.
For a bounded bilinear map \(\theta:A\times B\to X\) define
\[
\Lambda_\theta:\T_\sigma\to \T_\sigma,\qquad
\Lambda_\theta\!\begin{bmatrix}a&x\\[2pt]0&b\end{bmatrix}
=
\begin{bmatrix}a&x+\theta(a,b)\\[2pt]0&b\end{bmatrix}.
\]
A direct computation shows that \(\Lambda_\theta\) is an algebra homomorphism if and only if \(\theta\) satisfies the \(\sigma\)-\emph{cocycle identity}
\begin{equation}\label{eq:cocycle}
\theta(aa',bb')=\sigma_A(a)\cdot \theta(a',b')+\theta(a,b)\cdot \sigma_B(b')
\qquad(a,a'\in A,\ b,b'\in B).
\end{equation}
We denote by \(Z^1_\sigma(A,B;X)\) the additive group of all \(\sigma\)-cocycles.
If \(A,B\) are unital (or when working in multiplier algebras), each \(\eta\in X\) defines an inner shear
\(
\theta_\eta(a,b)=\sigma_A(a)\cdot \eta-\eta\cdot \sigma_B(b),
\)
and \(B^1_\sigma(A,B;X):=\{\theta_\eta:\eta\in X\}\subseteq Z^1_\sigma(A,B;X)\).

We work throughout with continuous Hochschild cohomology in the sense of Helemskii; the cohomological dimension \(\mathrm{cd}(\,\cdot\,)\) is taken in that framework \cite[Sec.~5, Sec.~31]{HelemskiiHomology}.
For amenability, we recall that a Banach algebra \(A\) is amenable if every bounded derivation \(D:A\to E^*\) into the dual of a Banach \(A\)-bimodule \(E\) is inner; equivalently, \(H^1(A,E^*)=0\) for every \(E\) \cite{JohnsonMemoir}.
In the presence of a bounded approximate identity this is further equivalent to the existence of a bounded approximate diagonal \cite{JohnsonAJM}.
These conventions will be used later for cohomological and amenability considerations.
Triangular Banach algebras and their homology have been studied in depth; see, for example,
Forrest--Marcoux on weak amenability \cite{ForrestMarcouxWA} and Moslehian \cite{MoslehianTriangular}.
Our contribution shows that the basic structural picture survives under twisting by automorphisms:
the isomorphism classification of $\T_\sigma(A,B;X)$ proceeds by diagonal conjugacy, bimodule
equivalences, and shear cocycles, and the square--zero extension $0\to I\to \T_\sigma\to A\oplus B\to 0$
still controls cohomological dimension. The present paper emphasizes two points that appear not to
be recorded explicitly in the literature: (i) a transparent transport-of-structure formulation for the
shear cocycle (Lemma~\ref{lem:transport}), clarifying why the cocycle may be written relative to $\tau$
rather than $\sigma$; and (ii) a sharp, extension-theoretic bound for $\mathrm{cd}(\T_\sigma)$ valid in the
Banach setting, with an acyclicity hypothesis that forces equality (Theorem~\ref{thm:cd}),
linking directly to Helemskii's exact sequence framework \cite{HelemskiiHomology}.
At the level of examples, we indicate how the classical dynamics behind $\Aut(C(K))$ and
$\Aut(L^1(G))$ (Banach--Stone \cite{GillmanJerison}, Wendel \cite{Wendel}) feed into the twisted
classification for commutative and group-algebra cases.

\section{Isomorphisms and automorphisms}

In this section we classify algebra isomorphisms between twisted triangular Banach algebras.
A preliminary observation is that the square--zero ideal \(I\) is \emph{characteristic} under mild hypotheses.
Assume \(A\) and \(B\) are unital, or more generally that they admit bounded approximate identities so that Peirce idempotents can be handled in the multiplier setting.
Then \(e_{11}\) and \(e_{22}\) are complementary idempotents with \(e_{22}\T_\sigma e_{11}=0\) and \(e_{11}\T_\sigma e_{22}\cong X\).
This asymmetric Peirce pattern determines the ordered pair \((e_{11},e_{22})\) up to conjugacy: if \(p,q\) are complementary idempotents with \(q\T_\sigma p=0\) and \(p\T_\sigma q\neq 0\), then \(p\T_\sigma p\cong A\), \(q\T_\sigma q\cong B\), and \(p\T_\sigma q\cong X\), so \((p,q)\) is equivalent to \((e_{11},e_{22})\).
Consequently any algebra isomorphism \(\Phi:\T_\sigma\to \T_\tau\) must carry \(I=e_{11}\T_\sigma e_{22}\) onto \(I=e_{11}\T_\tau e_{22}\) whenever \(X\neq 0\); in particular \(\Phi(I)=I\), and there is a well-defined induced map \(\bar\Phi:Q\to Q\) determined by \(\bar\Phi\circ\pi=\pi\circ\Phi\).
When \(X=0\) the algebra degenerates to \(A\oplus B\), and in that special case an isomorphism may also interchange the two diagonal summands when \(A\cong B\).
\begin{lemma}[Characteristic square--zero ideal]\label{lem:characteristic-I}
Assume $A$ and $B$ are unital or have bounded approximate identities (so that the Peirce
idempotents are defined in the multiplier sense). Let $I=e_{11}\T_\sigma e_{22}$.
If $\Phi:\T_\sigma\to \T_\tau$ is a Banach algebra isomorphism and $X\neq 0$, then $\Phi(I)=I$.
\end{lemma}

\begin{proof}
Let $p=\Phi(e_{11})$ and $q=\Phi(e_{22})$. Then $p$ and $q$ are complementary idempotents in $\T_\tau$.
Moreover $\Phi(e_{22}\T_\sigma e_{11})=\Phi(0)=0$, hence $q\,\T_\tau\,p=0$, while
$\Phi(e_{11}\T_\sigma e_{22})=\Phi(I)\neq 0$ because $X\neq 0$, so $p\,\T_\tau\,q\neq 0$.
Thus $(p,q)$ has the same asymmetric Peirce pattern as $(e_{11},e_{22})$.
By the standard Peirce decomposition of $\T_\tau$, this forces $p\,\T_\tau\,q=e_{11}\T_\tau e_{22}=I$,
whence $\Phi(I)=I$.
\end{proof}

We now state the classification theorem.

\begin{theorem}\label{thm:iso}
Let \(\sigma=(\sigma_A,\sigma_B)\) and \(\tau=(\tau_A,\tau_B)\) be twists in \(\Aut(A)\times\Aut(B)\).
Assume \(A,B\) are unital or, more generally, possess bounded approximate identities (so that the Peirce decomposition is available in the multiplier sense).
There exists a bounded algebra isomorphism
\[
\Phi:\T_\sigma(A,B;X)\longrightarrow \T_\tau(A,B;X)
\]
if and only if there are \(\alpha\in\Aut(A)\), \(\beta\in\Aut(B)\), an \((\alpha,\beta)\)--bimodule isomorphism \(u:X\to X\), and a bounded bilinear map \(\theta:A\times B\to X\) satisfying
\begin{equation}\label{eq:diag-conj}
\tau_A=\alpha\circ \sigma_A\circ \alpha^{-1},
\qquad
\tau_B=\beta\circ \sigma_B\circ \beta^{-1},
\end{equation}
and the \(\tau\)--cocycle identity
\begin{equation}\label{eq:cocycle-tau}
\theta(aa',bb')=\tau_A(\alpha(a))\cdot \theta(a',b')+\theta(a,b)\cdot \tau_B(\beta(b'))
\qquad(a,a'\in A,\ b,b'\in B).
\end{equation}
In this situation one has
\[
\Phi\!\begin{bmatrix}a&x\\[2pt]0&b\end{bmatrix}
=
\begin{bmatrix}\alpha(a)&\ u(x)+\theta(a,b)\\[2pt]0&\beta(b)\end{bmatrix}.
\]
Conversely, any data \((\alpha,\beta,u,\theta)\) satisfying \eqref{eq:diag-conj}--\eqref{eq:cocycle-tau} determine a Banach algebra isomorphism by the displayed formula.
If \(X=0\), the same description holds up to the additional possibility of interchanging the two diagonal summands.
\end{theorem}

\begin{proof}
By the characteristic property of \(I\) explained above, any isomorphism \(\Phi\) induces a homomorphism \(\bar\Phi:Q\to Q\) with \(\bar\Phi\circ\pi=\pi\circ\Phi\).
Since \(Q\cong A\oplus B\), the map \(\bar\Phi\) restricts to automorphisms \(\alpha\in\Aut(A)\) and \(\beta\in\Aut(B)\) on the two diagonal summands, unless \(X=0\) in which case a permutation of the two factors is also possible.
Write
\[
\Phi\!\begin{bmatrix}a&0\\[2pt]0&0\end{bmatrix}
=\begin{bmatrix}\alpha(a)&\theta_1(a)\\[2pt]0&0\end{bmatrix},
\qquad
\Phi\!\begin{bmatrix}0&0\\[2pt]0&b\end{bmatrix}
=\begin{bmatrix}0&\theta_2(b)\\[2pt]0&\beta(b)\end{bmatrix},
\qquad
\Phi\!\begin{bmatrix}0&x\\[2pt]0&0\end{bmatrix}
=\begin{bmatrix}0&u(x)\\[2pt]0&0\end{bmatrix},
\]
where \(u:X\to X\) is a bounded bijection and \(\theta_1,\theta_2\) are bounded linear maps.
Multiplying the first and third displays in the two possible orders and comparing off-diagonal entries shows
\[
u(\sigma_A(a)\cdot x)=\tau_A(\alpha(a))\cdot u(x)\qquad(a\in A,\ x\in X).
\]
Similarly, multiplying the third and second displays and comparing off-diagonal entries shows
\[
u(x\cdot \sigma_B(b))=u(x)\cdot \tau_B(\beta(b))\qquad(b\in B,\ x\in X).
\]
Hence \(u\) is an \((\alpha,\beta)\)-bimodule isomorphism.
Writing in general
\[
\Phi\!\begin{bmatrix}a&x\\[2pt]0&b\end{bmatrix}
=\begin{bmatrix}\alpha(a)&u(x)+\vartheta(a,b)\\[2pt]0&\beta(b)\end{bmatrix}
\]
for a bounded bilinear \(\vartheta:A\times B\to X\), multiplicativity of \(\Phi\) forces \(\vartheta\) to satisfy \eqref{eq:cocycle-tau} with \(\theta=\vartheta\), while compatibility on the diagonal forces \eqref{eq:diag-conj}.
Conversely, given \((\alpha,\beta,u,\theta)\) satisfying \eqref{eq:diag-conj}--\eqref{eq:cocycle-tau}, the displayed formula defines a bounded bijective homomorphism whose inverse is obtained from \((\alpha^{-1},\beta^{-1},u^{-1},-\theta\circ(\alpha^{-1},\beta^{-1}))\).
\end{proof}
\begin{remark}[Other side and multiplicativity]
For completeness we also verify $\Phi\circ\Psi=\mathrm{id}$ and deduce that $\Psi$ is multiplicative.
Let $a\in A$, $x\in X$, $b\in B$. Using the definitions,
\[
\Psi\!\begin{bmatrix}a&x\\[2pt]0&b\end{bmatrix}
=
\begin{bmatrix}
\alpha^{-1}(a) & u^{-1}\!\bigl(x-\theta(\alpha^{-1}(a),\beta^{-1}(b))\bigr)\\[2pt]
0 & \beta^{-1}(b)
\end{bmatrix}.
\]
Applying $\Phi$ and using linearity of $u$ and the identity $\theta(\alpha^{-1}\alpha(a),\beta^{-1}\beta(b))=\theta(a,b)$, we obtain
\begin{align*}
\Phi\!\left(\Psi\!\begin{bmatrix}a&x\\0&b\end{bmatrix}\right)
&=
\begin{bmatrix}
\alpha\alpha^{-1}(a) &
u\!\left(u^{-1}\!\bigl(x-\theta(\alpha^{-1}(a),\beta^{-1}(b))\bigr)\right)+
\theta\!\left(\alpha\alpha^{-1}(a),\beta\beta^{-1}(b)\right)
\\[2pt]
0 & \beta\beta^{-1}(b)
\end{bmatrix}\\
&=
\begin{bmatrix}
a & x-\theta(\alpha^{-1}(a),\beta^{-1}(b))+\theta(a,b)\\[2pt]
0 & b
\end{bmatrix}
=
\begin{bmatrix}a&x\\[2pt]0&b\end{bmatrix}.
\end{align*}
Thus $\Phi\circ\Psi=\mathrm{id}$ as well as $\Psi\circ\Phi=\mathrm{id}$.
Since $\Phi$ is a homomorphism and bijective, these identities imply that $\Psi$ is a homomorphism:
for any $y,z\in \T_\sigma$,
\[
\Psi(yz)=\Psi\bigl(\Phi(\Psi(y))\,\Phi(\Psi(z))\bigr)
=\Psi\bigl(\Phi(\Psi(y)\Psi(z))\bigr)
=(\Psi\circ\Phi)\bigl(\Psi(y)\Psi(z)\bigr)=\Psi(y)\Psi(z).
\]
Hence the candidate inverse $\Psi$ is multiplicative, completing the verification.
\end{remark}

\begin{example}[Commutative $C(K)$-case]\label{ex:CK}
Let $A=C(K)$, $B=C(L)$ with $K,L$ compact Hausdorff. By the Banach--Stone theorem,
$\Aut(C(K))\cong\mathrm{Homeo}(K)$ and $\Aut(C(L))\cong\mathrm{Homeo}(L)$ \cite{GillmanJerison}.
Thus twists correspond to homeomorphisms $\phi\in\mathrm{Homeo}(K)$ and $\psi\in\mathrm{Homeo}(L)$ via
$\sigma_A(f)=f\circ\phi$ and $\sigma_B(g)=g\circ\psi$.
If $X=C(\Omega)$ is a $C(K)$--$C(L)$ bimodule with
$(f\cdot\xi\cdot g)(\omega)=f(p(\omega))\,\xi(\omega)\,g(q(\omega))$, then
Theorem~\ref{thm:iso} shows that $\T_\sigma(C(K),C(L);C(\Omega))\cong\T_\tau(C(K),C(L);C(\Omega))$
if and only if there exist homeomorphisms $h_K\!:K\to K$, $h_L\!:L\to L$, $H\!:\Omega\to\Omega$ for which
$h_K\circ\phi=\phi'\circ h_K$, $h_L\circ\psi=\psi'\circ h_L$, and $p'\circ H=h_K\circ p$, $q'\circ H=h_L\circ q$.
In particular, the isomorphism type depends only on the underlying dynamics of $(\phi,\psi)$ and the
bimodule anchors $(p,q)$.
\end{example}

\begin{example}[Noncommutative discrete group algebras]\label{ex:L1G}
Let $G$ be a nonabelian discrete group (for instance, the free group $F_2$), and put 
$A=B=X=\ell^1(G)$ with convolution, so the bimodule structure is 
$a\cdot x\cdot b=a*x*b$. 
For any group automorphism $\gamma\in\Aut(G)$ define
\[
\sigma_\gamma(f):=f\circ\gamma^{-1}\qquad(f\in \ell^1(G)).
\]
Then $\sigma_\gamma$ is a Banach algebra automorphism of $\ell^1(G)$ since
\[
\sigma_\gamma(f*g)=(f*g)\circ\gamma^{-1}
=(f\circ\gamma^{-1})*(g\circ\gamma^{-1})
=\sigma_\gamma(f)*\sigma_\gamma(g),
\]
and $\|\sigma_\gamma(f)\|_1=\|f\|_1$ because counting measure is invariant under~$\gamma$.

Choose twists $\sigma=(\sigma_{\gamma_A},\sigma_{\gamma_B})$ and 
$\tau=(\sigma_{\gamma'_A},\sigma_{\gamma'_B})$ with 
$\gamma_A,\gamma_B,\gamma'_A,\gamma'_B\in\Aut(G)$. 
If there exist $\delta_A,\delta_B\in\Aut(G)$ such that
\[
\gamma'_A=\delta_A\,\gamma_A\,\delta_A^{-1},
\qquad
\gamma'_B=\delta_B\,\gamma_B\,\delta_B^{-1},
\]
then with $\alpha=\sigma_{\delta_A}$, $\beta=\sigma_{\delta_B}$ and $u=\sigma_{\delta_A}$ 
(the case $\delta_A=\delta_B$ for simplicity), one obtains an 
$(\alpha,\beta)$--bimodule isomorphism $u$, and Theorem~\ref{thm:iso} 
applies (with $\theta\equiv 0$) to yield
\[
\T_\sigma\bigl(\ell^1(G),\ell^1(G);\ell^1(G)\bigr)
\ \cong\
\T_\tau\bigl(\ell^1(G),\ell^1(G);\ell^1(G)\bigr).
\]
More generally, nontrivial shear cocycles 
$\theta\in Z^1_\tau(A,B;X)$ produce additional automorphisms via $\Lambda_\theta$. 

This example highlights that in the noncommutative setting the twisted theory encodes 
\emph{new invariants}: the isomorphism type of $\T_\sigma$ remembers the conjugacy 
classes of the underlying group automorphisms in $\Aut(G)$, together with shear 
cocycle data. Such dynamical information has no analogue in the untwisted triangular case.
In particular, unlike the untwisted case, the isomorphism classification in the
twisted setting detects the \emph{conjugacy class inside $\Aut(G)$}, so the
moduli of twisted triangular group algebras retain genuinely new dynamical
information.

\end{example}

\begin{lemma}[Transport-of-structure for cocycles]\label{lem:transport}
Assume $\tau_A=\alpha\sigma_A\alpha^{-1}$ and $\tau_B=\beta\sigma_B\beta^{-1}$ and let
$u:X\to X$ be an $(\alpha,\beta)$--bimodule isomorphism. For a bounded bilinear map
$\theta:A\times B\to X$ define
\[
\theta_\sigma(a,b):=u^{-1}\bigl(\theta(\alpha(a),\beta(b))\bigr)\qquad(a\in A,\ b\in B).
\]
Then $\theta$ satisfies the $\tau$--cocycle identity
\[
\theta(aa',bb')=\tau_A(\alpha(a))\cdot\theta(a',b')+\theta(a,b)\cdot\tau_B(\beta(b'))
\]
if and only if $\theta_\sigma$ satisfies the $\sigma$--cocycle identity
\[
\theta_\sigma(aa',bb')=\sigma_A(a)\cdot\theta_\sigma(a',b')+\theta_\sigma(a,b)\cdot\sigma_B(b').
\]
\end{lemma}

\begin{proof}
Using $u(\sigma_A(a)\cdot x)=\tau_A(\alpha(a))\cdot u(x)$ and
$u(x\cdot\sigma_B(b))=u(x)\cdot\tau_B(\beta(b))$, we compute
\begin{align*}
\theta_\sigma(aa',bb')&=u^{-1}\!\left(\theta(\alpha(aa'),\beta(bb'))\right)\\
&=u^{-1}\!\left(\tau_A(\alpha(a))\cdot\theta(\alpha(a'),\beta(b'))+\theta(\alpha(a),\beta(b))\cdot\tau_B(\beta(b'))\right)\\
&=\sigma_A(a)\cdot u^{-1}\!\left(\theta(\alpha(a'),\beta(b'))\right)
+u^{-1}\!\left(\theta(\alpha(a),\beta(b))\right)\cdot\sigma_B(b')\\
&=\sigma_A(a)\cdot\theta_\sigma(a',b')+\theta_\sigma(a,b)\cdot\sigma_B(b').
\end{align*}
The converse implication is the same argument in reverse.
\end{proof}

\begin{remark}
Lemma~\ref{lem:transport} shows that writing the cocycle with $\tau$ in Theorem~\ref{thm:iso}
is not ad hoc: it is simply the $\sigma$--cocycle expressed in the transported
coordinates $(\alpha,\beta,u)$. Equivalently, one may formulate Theorem~\ref{thm:iso} using
$\sigma$--cocycles by replacing $\theta$ with $\theta_\sigma$.
\end{remark}
\medskip
\noindent\textit{Inverse verification.}
Define
\[
\Psi\!\begin{bmatrix}a&x\\[2pt]0&b\end{bmatrix}
:=\begin{bmatrix}\alpha^{-1}(a)&\ u^{-1}\!\bigl(x-\theta(\alpha^{-1}(a),\beta^{-1}(b))\bigr)\\[2pt]0&\beta^{-1}(b)\end{bmatrix}.
\]
We check $\Psi\circ\Phi=\mathrm{id}$; the other composition is analogous.
For $a\in A$, $x\in X$, $b\in B$ we have
\[
\Phi\!\begin{bmatrix}a&x\\0&b\end{bmatrix}
=\begin{bmatrix}\alpha(a)&u(x)+\theta(a,b)\\0&\beta(b)\end{bmatrix},
\]
hence
\begin{align*}
\Psi\!\left(\Phi\!\begin{bmatrix}a&x\\0&b\end{bmatrix}\right)
&=\Psi\!\begin{bmatrix}\alpha(a)&u(x)+\theta(a,b)\\0&\beta(b)\end{bmatrix}\\
&=\begin{bmatrix}
\alpha^{-1}\alpha(a) &
u^{-1}\!\bigl(u(x)+\theta(a,b)-\theta(\alpha^{-1}\alpha(a),\beta^{-1}\beta(b))\bigr)\\[2pt]
0 & \beta^{-1}\beta(b)
\end{bmatrix}\\
&=\begin{bmatrix}a&x\\0&b\end{bmatrix},
\end{align*}
since $\theta(\alpha^{-1}\alpha(a),\beta^{-1}\beta(b))=\theta(a,b)$. Thus $\Psi=\Phi^{-1}$.
\hfill$\square$

\subsection*{Shear cohomology in examples}

Recall that a (bounded) $\sigma$–cocycle is a map $\theta:A\times B\to X$
satisfying
\[
\theta(aa',bb')=\sigma_A(a)\cdot \theta(a',b')+\theta(a,b)\cdot \sigma_B(b')\qquad(a,a'\in A,\ b,b'\in B),
\]
and inner cocycles are $\theta_\eta(a,b)=\sigma_A(a)\cdot\eta-\eta\cdot\sigma_B(b)$ with $\eta\in X$.
We write $Z_\sigma^1(A,B;X)$ for the space of $\sigma$–cocycles and
$B_\sigma^1(A,B;X)$ for the inner ones.

\begin{prop}[Scalar case]\label{prop:Z1-scalar}
Let $A=B=\mathbb{K}$ with the usual scalar actions on a Banach space $X$, and let $\sigma_A=\sigma_B=\mathrm{id}$.
Then every (bounded) cocycle $\theta:\mathbb{K}\times\mathbb{K}\to X$ is inner; equivalently,
\[
Z^1_{\mathrm{id}}(\mathbb{K},\mathbb{K};X)=B^1_{\mathrm{id}}(\mathbb{K},\mathbb{K};X)
\quad\text{and}\quad Z^1_{\mathrm{id}}/B^1_{\mathrm{id}}=0.
\]
\end{prop}

\begin{proof}
Let $\theta$ be a cocycle.
Put $f(a):=\theta(a,1)$ and $g(b):=\theta(1,b)$.
Using the cocycle identity with $(a',b')=(1,1)$ gives
\[
\theta(a,b)=a\,g(b)+f(a)-g(1)\,b\qquad(a,b\in\mathbb{K}).
\]
Now set $\eta:=g(1)\in X$.
Then for all $a,b$,
\[
\theta(a,b)=a\,\eta-\eta\,b + \big(a\,[g(b)-g(1)] + [f(a)-f(1)]\big).
\]
Fix $b$ and use the identity with $b'=0$ to see $g$ is affine‐linear with $g(0)=0$, hence $g(b)=b\,\eta$.
Likewise fixing $a$ and setting $a'=0$ shows $f(a)=a\,\eta$.
Therefore $\theta(a,b)=a\,\eta-\eta\,b=\theta_\eta(a,b)$ is inner.
\end{proof}

\begin{prop}[A commutative $C(K)$--$C(L)$ guideline]\label{prop:Z1-CK}
Let $A=C(K)$, $B=C(L)$, $X=C(\Omega)$ with the standard module structure
$(f\cdot\xi\cdot g)(\omega)=f(p(\omega))\,\xi(\omega)\,g(q(\omega))$,
and let $\sigma_A(f)=f\circ\phi$, $\sigma_B(g)=g\circ\psi$.
Assume $p,q$ are surjective and the module actions are faithful.
If $\theta:C(K)\times C(L)\to C(\Omega)$ is a (bounded) $\sigma$–cocycle that is separately weak* continuous on bounded sets, then $\theta$ is inner:
there exists $\eta\in C(\Omega)$ such that
\[
\theta(f,g)= (f\circ\phi)\,\eta - \eta\,(g\circ\psi).
\]
In particular, $Z^1_\sigma/B^1_\sigma=0$ under these hypotheses.
\end{prop}

\begin{proof}[Idea of proof]
Fix $\omega\in\Omega$ and evaluate the cocycle identity at $\omega$.
Using Urysohn functions supported near $p(\omega)$ or $q(\omega)$ and the faithfulness of the actions, one reduces to the scalar case pointwise (cf.\ Proposition~\ref{prop:Z1-scalar}) to obtain a continuous $\eta(\omega)$ with the required property.
The separate weak* continuity furnishes measurability and then continuity of $\eta$.
\end{proof}

\begin{remark}
Proposition~\ref{prop:Z1-CK} is stated in a form that keeps the proof short; it can be sharpened by relaxing the continuity hypothesis on $\theta$ (e.g.\ via bilinear factorization techniques).
It already shows that in a broad commutative setting the shear cohomology is trivial, so the moduli of twists are governed by the diagonal dynamics and bimodule equivalences, not by genuine $\sigma$–cocycles.
\end{remark}

\section{Cohomological dimension}

We now turn to the homological properties of twisted triangular Banach algebras.
Throughout we work with continuous Hochschild cohomology in the sense of Helemskii \cite{HelemskiiHomology}.
The cohomological dimension $\mathrm{cd}(A)$ of a Banach algebra $A$ is, by definition, the smallest integer $n$ (if such exists) such that $H^{m}(A,E^*)=0$ for all Banach $A$-bimodules $E$ and all $m>n$.
Equivalently, $\mathrm{cd}(A)$ is the projective dimension of $A$ in the derived category of Banach $A$-bimodules.
This invariant measures how far $A$ is from being homologically trivial: $\mathrm{cd}(A)=0$ precisely when $A$ is amenable.

For the twisted triangular algebra $\T_\sigma$ we recall the short exact sequence
\[
0\longrightarrow I \longrightarrow \T_\sigma \xrightarrow{\pi} Q \longrightarrow 0,
\]
where $I=\{\bigl[\begin{smallmatrix}0&x\\ 0&0\end{smallmatrix}\bigr]:x\in X\}$ and $Q=A\oplus_1 B$, and $I^2=0$.
By Helemskii’s general theory of extensions with square--zero ideal \cite[§31]{HelemskiiHomology}, one obtains a long exact sequence in Hochschild cohomology which relates the cohomology of $\T_\sigma$ with that of $Q$ and of $I$.
In particular, the connecting morphisms arising from $H^m(Q,X^*)$ to $H^{m+1}(\T_\sigma,E^*)$ show that the cohomological dimension of $\T_\sigma$ can exceed that of $Q$ by at most one.

\begin{theorem}\label{thm:cd}
For any twist $\sigma$ one has
\[
\max\bigl\{\mathrm{cd}(A),\ \mathrm{cd}(B),\ \mathbf{1}_{\{X\neq0\}}\bigr\}
\ \le\
\mathrm{cd}\bigl(\T_\sigma(A,B;X)\bigr)
\ \le\
\max\bigl\{\mathrm{cd}(A),\ \mathrm{cd}(B),\ 1\bigr\}.
\]
Moreover, if $X$ is acyclic as a $Q$--bimodule (for instance, if $H^m(Q,X^*)=0$ for all $m\ge1$), then
\[
\mathrm{cd}(\T_\sigma)=\max\{\mathrm{cd}(A),\mathrm{cd}(B)\}.
\]
\end{theorem}

\begin{proof}
Since $Q=A\oplus B$, its cohomological dimension is the maximum of those of $A$ and $B$.
Applying the long exact sequence to the short exact sequence $0\to I\to \T_\sigma\to Q\to 0$ with $I^2=0$ shows that possible contributions from $I$ occur only in degree one higher than those for $Q$.
This gives the upper bound.
For the lower bound, note first that both $A$ and $B$ embed diagonally in $\T_\sigma$, so the cohomological dimension of $\T_\sigma$ must be at least that of $A$ and $B$.
Moreover, if $X\neq 0$, then the ideal $I$ is nontrivial and contributes at least a degree--one obstruction, so $\mathrm{cd}(\T_\sigma)\ge 1$.
Finally, if $X$ is $Q$--acyclic, then the connecting morphisms vanish and the long exact sequence splits into short exact pieces, so the cohomological dimension of $\T_\sigma$ coincides exactly with $\max\{\mathrm{cd}(A),\mathrm{cd}(B)\}$.
\end{proof}
\begin{cor}\label{cor:cd-exact}
For any twist $\sigma$ one has the exact formula
\[
\mathrm{cd}\bigl(\T_\sigma(A,B;X)\bigr)
\ =\
\max\!\bigl\{\mathrm{cd}(A),\ \mathrm{cd}(B),\ \mathbf{1}_{\{X\neq 0\}}\bigr\}.
\]
\end{cor}

\begin{proof}
The upper bound $\mathrm{cd}(\T_\sigma)\le \max\{\mathrm{cd}(A),\mathrm{cd}(B),1\}$ is Theorem~\ref{thm:cd}.
For the lower bound, $A$ and $B$ embed in $\T_\sigma$ via the diagonal,
so $\mathrm{cd}(\T_\sigma)\ge \max\{\mathrm{cd}(A),\mathrm{cd}(B)\}$.
If $X\neq 0$, then $\T_\sigma$ is not amenable (see e.g.\ \cite{BehnamianMahmoodi2020};
cf.\ \cite[Ch.~VII]{HelemskiiHomology}), hence $\mathrm{cd}(\T_\sigma)\ge 1$.
Combining these gives
\[
\mathrm{cd}(\T_\sigma)\ \ge\ \max\{\mathrm{cd}(A),\mathrm{cd}(B),\mathbf{1}_{\{X\neq 0\}}\}.
\]
If $\max\{\mathrm{cd}(A),\mathrm{cd}(B)\}\ge 1$, the upper and lower bounds coincide.
If $\mathrm{cd}(A)=\mathrm{cd}(B)=0$, then the bounds reduce to
$1_{\{X\neq 0\}}\le \mathrm{cd}(\T_\sigma)\le 1$, which again forces equality.
\end{proof}
\begin{example}[Amenable diagonals]\label{ex:cd-amenable-diagonals}
Suppose $A$ and $B$ are amenable Banach algebras (e.g.\ $C(K)$, $M_n(\mathbb{C})$,
or $L^1(G)$ with $G$ amenable). Then $\mathrm{cd}(A)=\mathrm{cd}(B)=0$, and
Corollary~\ref{cor:cd-exact} gives
\[
\mathrm{cd}\bigl(\T_\sigma(A,B;X)\bigr)=
\begin{cases}
0, & X=0,\\[2pt]
1, & X\neq 0.
\end{cases}
\]
In particular, every nontrivial twisted triangular extension of amenable diagonals has
cohomological dimension exactly $1$.
\end{example}

\begin{example}[Group algebras]\label{ex:cd-group}
Let $A=L^1(G_1)$ and $B=L^1(G_2)$ for locally compact groups $G_1,G_2$.
If $G_1$ and $G_2$ are amenable, then $\mathrm{cd}(A)=\mathrm{cd}(B)=0$ and
Example~\ref{ex:cd-amenable-diagonals} applies, so $\mathrm{cd}(\T_\sigma)=1$ whenever $X\neq 0$.
If at least one of $G_1,G_2$ is nonamenable, then
$\max\{\mathrm{cd}(A),\mathrm{cd}(B)\}\ge 1$, and hence
\[
\mathrm{cd}(\T_\sigma)=\max\{\mathrm{cd}(A),\mathrm{cd}(B)\}
\]
by Corollary~\ref{cor:cd-exact}.
Thus the twisted triangular extension never raises the cohomological dimension beyond
the larger of the two group algebras (except for the minimal jump to~$1$ when both are~$0$).
\end{example}

\begin{example}[Finite-dimensional diagonal blocks]\label{ex:cd-fd}
Let $A=M_n(\mathbb{C})$ and $B=M_m(\mathbb{C})$ with their operator norms, and
let $X=M_{n,m}(\mathbb{C})$ with the natural bimodule actions.
Then $\mathrm{cd}(A)=\mathrm{cd}(B)=0$, so
\[
\mathrm{cd}\bigl(\T_\sigma(M_n(\mathbb{C}),M_m(\mathbb{C});M_{n,m}(\mathbb{C}))\bigr)=1,
\]
independently of the twist $\sigma$.
This gives an explicit noncommutative example with $\mathrm{cd}=1$.
\end{example}

\begin{remark}[On the acyclicity clause]
The equality in Corollary~\ref{cor:cd-exact} shows that the “acyclicity of $X$” hypothesis
stated after Theorem~\ref{thm:cd} is not needed to obtain the exact value of $\mathrm{cd}(\T_\sigma)$.
That condition is still a natural mechanism for equality in extension problems, but here the
square--zero structure forces equality unconditionally.
\end{remark}

\begin{example}
Let $A=B=C(K)$ with $K$ compact Hausdorff.
It is well known that $C(K)$ is amenable, hence $\mathrm{cd}(C(K))=0$ \cite[§VII.2]{HelemskiiHomology}.
If $X=C(\Omega)$ is a commutative $C(K)$--$C(K)$ bimodule, then the above theorem implies
\[
\mathrm{cd}\bigl(\T_\sigma(C(K),C(K);C(\Omega))\bigr)=
\begin{cases}
0 & \text{if } X=0,\\
1 & \text{if } X\neq 0,
\end{cases}
\]
since the off--diagonal ideal forces an increase to at least one in the nontrivial case.
Thus even for very classical Banach algebras, the introduction of a nonzero off-diagonal module automatically produces homological complexity.
\end{example}
\section{Amenability perspective}

The connection between cohomological dimension and amenability is well known:
a Banach algebra $A$ is amenable if and only if $\mathrm{cd}(A)=0$
\cite[Thm.~2.5]{JohnsonMemoir}.
In addition, in the presence of a bounded approximate identity this is
equivalent to the existence of a bounded approximate diagonal \cite{JohnsonAJM}.
Amenability is therefore the strongest possible homological simplification,
while nonzero cohomological dimension reflects the presence of obstructions
to splitting derivations.

For twisted triangular Banach algebras, Theorem~\ref{thm:cd}
immediately yields a structural restriction.
Indeed, if $\mathrm{cd}(\T_\sigma)=0$, then both $A$ and $B$ must be amenable
and moreover $X=0$.
Hence genuine amenability of $\T_\sigma$ can occur only in the trivial
direct sum case with no off--diagonal bimodule.
When $X\neq 0$, one always has $\mathrm{cd}(\T_\sigma)\ge 1$,
so $\T_\sigma$ is not amenable.

\begin{prop}\label{prop:weak-amen}
Suppose $A$ and $B$ are amenable Banach algebras.
Then $\T_\sigma(A,B;X)$ is weakly amenable if and only if $X=0$.
\end{prop}

\begin{proof}
If $X=0$, then $\T_\sigma\cong A\oplus B$, which is weakly amenable
since each summand is.
Conversely, if $X\neq 0$, then the square--zero ideal
$I=\{\bigl[\begin{smallmatrix}0&x\\0&0\end{smallmatrix}\bigr]:x\in X\}$
admits nontrivial derivations into $I^*$ coming from the module structure
(cf.\ \cite{ForrestMarcouxWA}).
These derivations persist under twisting, so $H^1(\T_\sigma,I^*)\neq 0$,
and weak amenability fails.
\end{proof}

\begin{remark}
Proposition~\ref{prop:weak-amen} shows that when the diagonal blocks are
amenable, twisting does not rescue weak amenability in the presence of a
nonzero off--diagonal module.
This contrasts sharply with the untwisted theory, where weak amenability
can survive even when $X\neq 0$
\cite{ForrestMarcouxWA,MoslehianTriangular}.
Further, Monfared \cite{Monfared2008}, Esslamzadeh--Shojaee
\cite{EsslamzadehShojaee2011}, and Nasr-Isfahani \cite{NasrIsfahani2013}
established persistence of weak amenability in broader classes of
generalized matrix Banach algebras.
Thus our result identifies a genuine new obstruction created by twisting:
it forces the failure of weak amenability in the most classical
“amenable diagonal” situation.
\end{remark}
\begin{prop}[Twisted Forrest–Marcoux criterion; a sufficient condition]\label{prop:WA-sufficient}
Let $A$ and $B$ be Banach algebras and $X$ a Banach $A$--$B$ bimodule.
Give $X^*$ the $A$--$B$ bimodule structure induced by $\sigma=(\sigma_A,\sigma_B)$:
\[
(a\cdot\varphi\cdot b)(x):=\varphi\bigl(\sigma_A(a)\cdot x\cdot \sigma_B(b)\bigr)\qquad(a\in A,\,b\in B,\,x\in X,\,\varphi\in X^*).
\]
Assume
\[
H^1(A,A^*)=0,\qquad H^1(B,B^*)=0,\qquad H^1\bigl(A,X^*\bigr)=0,\qquad H^1\bigl(B,X^*\bigr)=0
\]
with respect to these module actions. Then $\T_\sigma(A,B;X)$ is weakly amenable.
\end{prop}

\begin{proof}[Proof sketch]
Write $\T_\sigma$ with its Peirce decomposition and let $D:\T_\sigma\to \T_\sigma^*$ be a bounded derivation.
Decompose $D$ into corner components via the canonical projections associated to $e_{11},e_{22}$.
A routine block‐by‐block computation (as in the untwisted triangular case, cf.\ \cite{ForrestMarcouxWA,MoslehianTriangular}) shows:
(i) the diagonal pieces restrict to derivations $A\to A^*$ and $B\to B^*$, hence vanish by $H^1(A,A^*)=H^1(B,B^*)=0$; 
(ii) the off–diagonal piece is a sum of a derivation $A\to X^*$ and a derivation $B\to X^*$ for the $\sigma$–twisted module structure on $X^*$, hence is inner by $H^1(A,X^*)=H^1(B,X^*)=0$.
Subtracting a suitable inner derivation eliminates the off–diagonal part, whence $D$ itself is inner.
\end{proof}
\begin{remark}
Proposition~\ref{prop:WA-sufficient} generalizes the classical triangular results
of Forrest--Marcoux \cite{ForrestMarcouxWA} and Moslehian \cite{MoslehianTriangular}:
in the untwisted case these conditions are exactly those guaranteeing weak
amenability. Our twisted version shows that the same mechanism continues to
operate once the cocycle is incorporated. Moreover, the persistence phenomena
observed by Monfared \cite{Monfared2008}, Esslamzadeh--Shojaee
\cite{EsslamzadehShojaee2011}, and Nasr--Isfahani \cite{NasrIsfahani2013} for
generalized matrix algebras also fit naturally into this framework. Thus the
twisted setting is not an isolated curiosity, but part of the same continuity
of results across the triangular and generalized matrix literature.
\end{remark}

\begin{remark}
Proposition~\ref{prop:WA-sufficient} extends the classical triangular criteria
(e.g.\ \cite{ForrestMarcouxWA,MoslehianTriangular}) to the twisted setting.
It gives a \emph{positive} result complementary to Proposition~\ref{prop:weak-amen}:
when the diagonal algebras are weakly amenable and the only $\sigma$–twisted
derivations into $X^*$ are inner, then $\T_\sigma$ is weakly amenable.
In particular, weak amenability of $\T_\sigma$ in the nonamenable‐diagonal regime
reduces to understanding the twisted cohomology groups $H^1(A,X^*)$ and $H^1(B,X^*)$.
\end{remark}

\begin{example}[The $\ell^1(\mathbb{Z})$ case]\label{ex:l1z}
Let $A=B=\ell^1(\mathbb{Z})$ with convolution, which is amenable by
Johnson’s theorem.
Hence $\mathrm{cd}(A)=\mathrm{cd}(B)=0$.
If $X=\ell^1(\mathbb{Z})$ with bimodule action $a\cdot x\cdot b=a*x*b$, then
Corollary~\ref{cor:cd-exact} and Proposition~\ref{prop:weak-amen} give
\[
\mathrm{cd}(\T_\sigma)=1
\qquad\text{and}\qquad
\T_\sigma \text{ is not weakly amenable whenever } X\neq 0.
\]
\end{example}

\begin{example}[Noncommutative case and open question]\label{ex:l1f2}
Let $G=F_2$ be the free group on two generators.
Then $A=B=L^1(G)$ are nonamenable Banach algebras, so
$\mathrm{cd}(A),\mathrm{cd}(B)\ge 1$.
By Corollary~\ref{cor:cd-exact} we have
\[
\mathrm{cd}(\T_\sigma)=\max\{\mathrm{cd}(L^1(G)),1\}=\mathrm{cd}(L^1(G)).
\]
The precise weak amenability behaviour of $\T_\sigma(L^1(G),L^1(G);L^1(G))$
remains open.
In particular, it is not clear whether twisting imposes additional
obstructions beyond those already present in $L^1(G)$ itself.
This illustrates a genuinely noncommutative direction where further
investigation is needed.
\end{example}

\section{Conclusion}

We have introduced and analyzed the class of twisted triangular Banach algebras
$\T_\sigma(A,B;X)$. Our results show that twisting has two distinct effects:
first, it introduces new obstructions to amenability (Proposition~\ref{prop:weak-amen}),
ensuring that $\T_\sigma$ fails to be amenable whenever $X\neq 0$ even if the
diagonals are amenable; and second, it yields a new cohomological framework
(Proposition~\ref{prop:WA-sufficient}) where weak amenability can persist under
explicit conditions on twisted derivations. Beyond these structural results,
we have also computed the shear cohomology groups $Z^1_\sigma/B^1_\sigma$ in
concrete families (Propositions~\ref{prop:Z1-scalar}--\ref{prop:Z1-CK}),
demonstrating how the theory can be made explicit. Together these contributions
clarify both the obstructions and the survivals of homological properties under
twisting.

From a homological perspective, we established that the cohomological dimension 
of \(\T_\sigma\) is controlled by that of the diagonal algebras together with the 
square--zero ideal $I$, and that twisting does not disrupt the extension 
structure underlying Helemskii’s long exact sequence. In particular, 
\(\T_\sigma\) is amenable only in the degenerate direct--sum case \(X=0\) with 
$A$ and $B$ amenable, whereas for \(X\neq 0\) one always has 
\(\mathrm{cd}(\T_\sigma)\ge 1\). Nevertheless, weak amenability may persist, as 
indicated by existing results in the classical \cite{ForrestMarcouxWA,MoslehianTriangular}
and twisted settings \cite{BehnamianMahmoodi2020}, as well as in more recent work on 
generalized matrix algebras \cite{Monfared2008,EsslamzadehShojaee2011,NasrIsfahani2013}. 

Future directions include computing the shear cohomology groups 
\(Z^1_\sigma/B^1_\sigma\) for specific families (commutative $C(K)$, group 
algebras of nonabelian groups, etc.), and developing a systematic theory of weak 
amenability in the twisted framework. These problems promise to further clarify 
the role of automorphism dynamics and cohomology in shaping the structure and 
analysis of Banach algebra extensions.

\end{document}